\documentclass[10pt]{article}
\usepackage{amsmath,amssymb,amsthm}
\usepackage{graphicx} 
\usepackage{slashbox}
\usepackage{tikz} 
\usepackage{hyperref}
\usetikzlibrary{arrows,calc}

\newtheorem{prop}{Proposition}
\newtheorem{thm}{Theorem}

\theoremstyle{remark}
\newtheorem{rmk}{Remark}
\newtheorem{ex}{Example}

\begin{document}

\title{Hafnian of two-parameter matrices}
\author{Dmitry Efimov\thanks{e-mail: dmefim@mail.ru}\\ Institute of Physics and Mathematics,\\ Komi Science Centre UrD RAS,\\
Syktyvkar, Russia}
\date{}

\maketitle

\begin{abstract}
The concept of the hafnian first appeared in the works on quantum field theory by E. R. Caianiello.
However, it also has an important combinatorial property: the hafnian of the adjacency matrix of an undirected weighted graph
is equal to the total sum of the weights of perfect matchings in this graph.
In general, the use of the hafnian is limited by the complexity of its computation.
In this paper, we present an efficient method for the exact calculation of the hafnian of two-parameter matrices. 
In terms of graphs, we count the total sum of the weights of perfect matchings in graphs 
whose edge weights take only two values.
This method is based on the formula expressing the hafnian of a sum of two matrices through the product of the hafnians of their submatrices.
The necessary condition for the application of this method is the possibility to count the number of $k$-edge matchings in some graphs.
We consider two special cases in detail using a Toeplitz matrix as the two-parameter matrix.
As an example, we propose a new interpretation of some of the sequences from the On-Line Encyclopedia of Integer Sequences
and then provide new analytical formulas to count the number of certain linear chord diagrams. 
\end{abstract}

\section*{Introduction} 
Let $A=(a_{ij})$ be a symmetric matrix of even order $n$ over a commutative associative ring.
The {\itshape hafnian} of $A$ is defined as 
$$
  \textrm{Hf}(A)=\sum_{(i_1i_2|\dots|i_{n-1}i_n)} a_{i_1i_2}\cdots a_{i_{n-1}i_n},
$$
where the sum ranges over all unordered partitions of the set $\{1,2,\dots, n\}$ into unordered pairs $(i_1i_2)$, $\dots$, $(i_{n-1}i_n)$.  
Therefore, if $n=4$, then $\textrm{Hf}(A)=a_{12}a_{34}+a_{13}a_{24}+a_{14}a_{23}$.
The hafnian of the empty matrix is considered as $1$.
Note that the elements of the main diagonal are not included in the definition of the hafnian.
For the sake of convenience, we assume that all matrices under consideration have a zero main diagonal.

A \textit{$k$-edge matching} in a graph is a set of its $k$ pairwise nonadjacent edges.
An $m$-edge matching in a graph with $2m$ vertices is called \textit{perfect matching}.
If a graph is weighted, then the \textit {weight} of the matching is a product of the weights of the edges
included in this matching.
The hafnian has a useful combinatorial property related to an important problem in graph theory and its applications: 
if $M$ is the adjacency matrix of an undirected weighted graph with an even number of vertices, 
then $\textrm{Hf}(M)$ equals the total sum of the weights  of  the perfect matchings in the graph.
Unfortunately, the widespread use of the hafnian is limited due to
the complexity of its computations in general.
For example, one of the fastest exact algorithms to compute the hafnian of an arbitrary complex $n\times n$ matrix 
runs in $O(n^32^{n/2})$ time, and, as the authors show, it seems to be close to an optimal one \cite{Titan}.

Because the calculation of the hafnian has a high computational complexity in general,
the actual problem is the discovery of efficient analytical formulas expressing the hafnian for special classes of matrices.
Let $T_n$ be a symmetric $(0,1)$-matrix of order $n$, and let $a$ and $b$ be elements of a ring $R$.
We denote a symmetric matrix of order $n$ by $T_n(a,b)$,
which is obtained from $T_n$ by replacing all instances of $1$ by $a$ 
and all zero elements outside the main diagonal by $b$.
For example (dots denote zeros),
$$
T_4=
\left(
\begin{array}{cccc}
\cdot&1&\cdot&1\\
1&\cdot&1&\cdot\\	
\cdot&1&\cdot&1\\
1&\cdot&1&\cdot
\end{array}
\right)\ \ \Longrightarrow\ \  
T_4(a,b)=
\left(
\begin{array}{cccc}
\cdot&a&b&a\\
a&\cdot&a&b\\	
b&a&\cdot&a\\
a&b&a&\cdot
\end{array}
\right).
$$
We can say that $T_n(a,b)$ is a {\itshape two-parameter} matrix, and $T_n$ is the {\itshape template} for $T_n(a,b)$. Note that $T_n(1,0)=T_n$.
In our work, we present an effective method for the exact computation of the hafnian of matrices $T_n(a,b)$.
In terms of graphs, we count the total sum of weights of perfect matchings in two-parameter weighted graphs 
(i.e., weights of the edges are only $a$ and $b$).
This method is based on the formula expressing the hafnian of a sum of two matrices through the sum of the product of the hafnians of matrices
and is also closely linked to the combinatorial problem of counting the number of $k$-edge matchings in graphs.
In theoretical physics, this problem is known as the {\itshape monomer-dimer problem} \cite{Grimson}.

\begin{figure}[!ht]
\centering

\begin{tikzpicture}[scale=0.8]

\def\h{0}; 
\def\v{0};

\begin{scope}[thin]
 \coordinate (A1) at ($(1,2)+(\h,\v)$); 
 \coordinate (A2) at ($(0.5,1)+(\h,\v)$); 
 \coordinate (A3) at ($(1.5,1)+(\h,\v)$);
 \coordinate (A4) at ($(0,0)+(\h,\v)$);
 \coordinate (A5) at ($(1,0)+(\h,\v)$);
 \coordinate (A6) at ($(2,0)+(\h,\v)$);
\end{scope}

\begin{scope}
 \draw (A1) -- (A4);
 \draw (A1) -- (A6);
 \draw (A2) -- (A5);
\end{scope}

\begin{scope}[thick, fill=black]
 \draw [fill] (A1) circle (0.06) node[above=1pt] {{\footnotesize\slshape A}};
 \draw [fill] (A2) circle (0.06) node[left=1pt] {{\footnotesize\slshape B}};
 \draw [fill] (A3) circle (0.06) node[right=1pt] {{\footnotesize\slshape C}};  
 \draw [fill] (A4) circle (0.06) node[below=1pt] {{\footnotesize\slshape D}}; 
 \draw [fill] (A5) circle (0.06) node[below=1pt] {{\footnotesize\textsl E}}; 
 \draw [fill] (A6) circle (0.06) node[below=1pt] {{\footnotesize\slshape F}}; 
\end{scope}

\def\h{3};
\def\v{0};

\begin{scope}[thin]
 \draw[dashed] ($(1,0)+(\h,\v)$) -- ($(6,0)+(\h,\v)$);
 \draw ($(1,0)+(\h,\v)$) arc (180:0:1 and 3/4); 
 \draw ($(1,0)+(\h,\v)$) arc (180:0:0.5 and 3/8);
 \draw ($(2,0)+(\h,\v)$) arc (180:0:1.5 and 9/8);
 \draw ($(2,0)+(\h,\v)$) arc (180:0:1 and 3/4); 
 \draw ($(3,0)+(\h,\v)$) arc (180:0:1.5 and 9/8); 
\end{scope}

\begin{scope} [thick, fill=black]
 \draw [fill] ($(1,0)+(\h,\v)$) circle (0.06) node[below=1pt] {{\footnotesize\slshape A}}; 
 \draw [fill] ($(2,0)+(\h,\v)$) circle (0.06) node[below=1pt] {{\footnotesize\slshape B}};
 \draw [fill] ($(3,0)+(\h,\v)$) circle (0.06) node[below=1pt] {{\footnotesize\slshape C}};  
 \draw [fill] ($(4,0)+(\h,\v)$) circle (0.06) node[below=1pt] {{\footnotesize\slshape D}}; 
 \draw [fill] ($(5,0)+(\h,\v)$) circle (0.06) node[below=1pt] {{\footnotesize\slshape E}}; 
 \draw [fill] ($(6,0)+(\h,\v)$) circle (0.06) node[below=1pt] {{\footnotesize\slshape F}}; 
\end{scope}

\end{tikzpicture}

\caption{A binary tree and its corresponding arc diagram}\label{pic_06.04_1}
\end{figure}
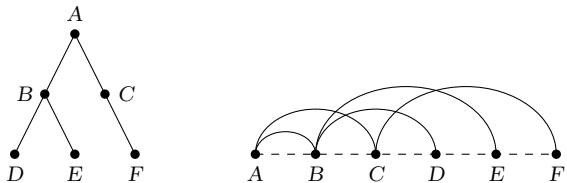

Recall that an \textit{arc diagram} is a graph presentation method in which all the vertices are located along a line in the plane,
whereas all edges are drawn as arcs (Figure \ref{pic_06.04_1}).
In this work, it will be convenient for us to represent graphs in the form of arc diagrams.
Perfect matchings of arc diagrams are often called {\itshape linear chord diagrams} \cite{Krasko, Sullivan}.

\section{Hafnian of two-parameter matrices}

To begin with, consider two properties of the hafnian.
The first one is quite obvious.

\begin{prop}\label{prop1}
 Let $A$ be a symmetric matrix of even order $n$ over a commutative associative ring $R$, and $c\in R$. Then
 \begin{equation}\label{16.10.18_1}
  \mathrm{Hf}(cA)=c^{n/2}\mathrm{Hf}(A).
 \end{equation}
\end{prop}

Let $Q_{k,n}$ denote the set of all unordered  $k$-element subsets of $\{1,2,\dots,n\}$.
Let $A$ be a matrix of order $n$ and $\alpha\in Q_{k,n}$. 
Moreover, $A[\alpha]$ denotes the submatrix of $A$ formed by the rows and columns of $A$ with numbers in $\alpha$, 
and $A\{\alpha\}$ denotes the submatrix of $A$ formed from $A$ by removing the rows and columns with numbers in $\alpha$.
The following property was proved in \cite{Efimov}:

\begin{prop}\label{prop2}
Let $A$ and $B$ be symmetric matrices of even order $n$. Then 
\begin{equation}\label{24.11_3}
 \mathrm{Hf}(A+B)=\sum_{k=0}^{n/2} \sum_{\alpha\in Q_{2k,n}}\mathrm{Hf}(A[\alpha])\mathrm{Hf}(B\{\alpha\}).
\end{equation}
\end{prop}

$J_n(b)$ denotes a matrix of order $n$ whose elements outside the main diagonal are equal to $b$.
From the definition of the hafnian, it follows that
\begin{equation}\label{08.04_1}
 \mathrm{Hf}(J_{2m}(b))=b^m\frac{(2m)!}{m!2^m}.
\end{equation}
Since $T_{2m}(a,b)=J_{2m}(b)+T_{2m}(a-b,0)$, using formulas
(\ref{16.10.18_1}), (\ref{24.11_3}), and (\ref{08.04_1}), we can write the following:
\begin{equation}\label{10.04.19_1}
\begin{split}
  &\mathrm{Hf}(T_{2m}(a,b))=\mathrm{Hf}(J_{2m}(b)+T_{2m}(a-b,0))=\\
	&=\sum_{k=0}^m \sum_{\alpha\in Q_{2k,2m}}\mathrm{Hf}(J_{2m}(b)[\alpha])\mathrm{Hf}(T_{2m}(a-b,0)\{\alpha\})=\\
	&=\sum_{k=0}^m (a-b)^{m-k}b^k\frac{(2k)!}{k!2^k}\sum_{\alpha\in Q_{2k,2m}}\mathrm{Hf}(T_{2m}\{\alpha\}).
\end{split}
\end{equation}
Here, we use the fact that the matrix $J_{2m}(b)[\alpha]$ has the same form as the initial matrix $J_{2m}(b)$,
that is,  $J_{2m}(b)[\alpha]$ is a matrix of order $2k$ whose elements outside the main diagonal are equal to $b$.
If $M$ is a symmetric nonnegative integer matrix, then $\Gamma(M)$ denotes a multigraph with the adjacency matrix $M$.
If $\alpha\in Q_{2k,2m}$, then the hafnian $\mathrm{Hf}(T_{2m}\{\alpha\})$ equals the cardinality
of a set of $(m-k)$-edge matchings in the graph $\Gamma(T_{2m})$;
moreover, such sets do not intersect for different $\alpha$, 
and their union is the set of all $(m-k)$-edge matchings of the graph $\Gamma(T_{2m})$.
Given a graph $\Gamma$, let $\mu_k(\Gamma)$ denote the number of all its $k$-edge matchings.
By definition, we set $\mu_0(\Gamma)=1$. 
Thus, 
$$
\sum_{\alpha\in Q_{2k,2m}}\mathrm{Hf}(T_{2m}\{\alpha\})=\mu_{m-k}(\Gamma(T_{2m})),
$$
and therefore,
\begin{equation}\label{11.04.19_1}
 \mathrm{Hf}(T_{2m}(a,b))=\sum_{k=0}^m (a-b)^{m-k}b^k\frac{(2k)!}{k!2^k}\mu_{m-k}(\Gamma(T_{2m})).
\end{equation}
Note that (\ref{11.04.19_1}) is the special case of Theorems 1W and 3W given
in \cite{Zaslavsky} in terms of matching vectors of weighted graphs and their complements.
The special case of (\ref{11.04.19_1}) when $a=0$ and $b=1$ is also given in \cite{Young} (Theorem $4$). 

Thus, to calculate the hafnian of a two-parameter matrix by using formula (\ref{11.04.19_1}), 
one needs to determine the number of $k$-edge matchings of  graphs corresponding to the matrix, 
which is a nontrivial task in general. 
One of the simplest special cases was considered in \cite{Efimov2}.
In the following section, we consider a few more complicated special cases.

\section{Hafnian of Toeplitz matrices of the first type}
Recall that a matrix is called {\itshape Toeplitz} if all its diagonals parallel
to the main diagonal consist of the same elements. It is obvious that a symmetric Toeplitz matrix
is uniquely determined by its first row.
As the template matrix $T_n$, consider a symmetric Toeplitz matrix of order $n$ with the first row
$$
\left(
\begin{array}{cccccc}
 0&0&1&0&\dots&0
\end{array}
\right).
$$
We denote it by $C_n$.
This matrix is the adjacency matrix of the arc diagram shown in Figure \ref{pic_13.01_21}.

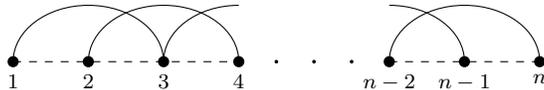
\begin{figure}[!ht]
\centering

\begin{tikzpicture}[scale=1]

\def\h{0}; 
\def\v{0};

\begin{scope}[thin]

 \draw[dashed] ($(1,0)+(\h,\v)$) -- ($(4,0)+(\h,\v)$);
 \draw[dashed] ($(6,0)+(\h,\v)$) -- ($(8,0)+(\h,\v)$);

 \draw ($(1,0)+(\h,\v)$) arc (180:0:1 and 3/4); 
 \draw ($(2,0)+(\h,\v)$) arc (180:0:1 and 3/4);
 \draw ($(3,0)+(\h,\v)$) arc (180:90:1 and 3/4);
 \draw ($(6,0)+(\h,\v)$) arc (180:0:1 and 3/4);
 \draw ($(7,0)+(\h,\v)$) arc (0:90:1 and 3/4);

\end{scope}

\begin{scope} [thick, fill=black]
 \draw [fill] ($(1,0)+(\h,\v)$) circle (0.06) node[below=1pt] {{\footnotesize 1}}; 
 \draw [fill] ($(2,0)+(\h,\v)$) circle (0.06) node[below=1pt] {{\footnotesize 2}}; 
 \draw [fill] ($(3,0)+(\h,\v)$) circle (0.06) node[below=1pt] {{\footnotesize 3}}; 
 \draw [fill] ($(4,0)+(\h,\v)$) circle (0.06) node[below=1pt] {{\footnotesize 4}}; 

 \draw [fill] ($(4.5,0)+(\h,\v)$) circle (0.01);  
 \draw [fill] ($(5,0)+(\h,\v)$) circle (0.01);  
 \draw [fill] ($(5.5,0)+(\h,\v)$) circle (0.01);  

 \draw [fill] ($(6,0)+(\h,\v)$) circle (0.06) node[below=1pt] {{\footnotesize $n-2$}};
 \draw [fill] ($(7,0)+(\h,\v)$) circle (0.06) node[below=1pt] {{\footnotesize $n-1$}}; 
 \draw [fill] ($(8,0)+(\h,\v)$) circle (0.06) node[below=1pt] {{\footnotesize $n$}}; 

\end{scope}

\end{tikzpicture}

\caption{The arc diagram $\Gamma(C_n)$}\label{pic_13.01_21}
\end{figure}

\begin{prop}\label{10.06.19_1}
Let $k$ and $n$ be a pair of nonnegative integers. 
Suppose $n\not= 2k$ when $k$ is odd.
Then, the inequality  
\begin{equation}\label{14.05.19_1}
	\left\lceil \frac{3k-n}{2}\right\rceil\leq \left\lfloor \frac{k}{2}\right\rfloor
\end{equation}
is equivalent to
\begin{equation}\label{14.05.19_2}
	k\leq \left\lfloor \frac{n}{2}\right\rfloor.
\end{equation}
\end{prop}

\begin{proof}
Suppose (\ref{14.05.19_2}) is wrong, which  means that $n=2k-c$, where $c\geq 1$.
On substituting this expression for $n$ into the inequality (\ref{14.05.19_1}), we can see that the result is wrong as well.
Now suppose (\ref{14.05.19_2}) is true.
Then, $n=2k+c$ and $\left\lceil \frac{3k-n}{2}\right\rceil=\left\lceil \frac{k-c}{2}\right\rceil$ , where $c\geq 0$.
If $k$ is even, then $\left\lceil \frac{k-c}{2}\right\rceil$ does not exceed $\left\lfloor \frac{k}{2}\right\rfloor$.
If $k$ is odd, then, by the assumption, $n\not=2k$.
Therefore, $c\geq 1$, and $\left\lceil \frac{k-c}{2}\right\rceil$ also does not exceed $\left\lfloor \frac{k}{2}\right\rfloor$.
Thus, the inequality (\ref{14.05.19_1}) also holds.
\end{proof}

\begin{prop}\label{12.04.19_3}
Let $k$ and $n$ be a pair of nonnegative integers. 
If  $k\leq \left\lfloor \frac{n}{2}\right\rfloor$, but $n\not=2k$ when $k$ is odd,
then the number of $k$-edge matchings in the arc diagram $\Gamma(C_n)$ is
\begin{equation}\label{19.08.19_1}
 \mu_k(\Gamma(C_n))=
\sum\limits_{i=\max{\left(0,\left\lceil \frac{3k-n}{2}\right\rceil\right)}}^{\left\lfloor \frac{k}{2}\right\rfloor}{n-2k+i\choose k-i}{k-i \choose i}.
\end{equation}
Otherwise, $\mu_k(\Gamma(C_n))=0$.
\end{prop}

\begin{proof}
For the sake of convenience, we use the abbreviated notation $v_{n,k}$ for $\mu_k(\Gamma(C_n))$.
Consider a $k$-edge matching in $\Gamma(C_n)$.
It is obvious that if $k>\left\lfloor \frac{n}{2}\right\rfloor$, then $v_{n,k}$=0.
If $n\geq 4$, then the following three cases are possible:
the first vertex of the diagram is not incident to the edge of the matching (Figure \ref{fig_13.01.21}(a));
the first vertex is incident to an edge of the matching, but the second vertex is not (Figure \ref{fig_13.01.21}(b)); 
the first and second vertices are incident to the edges of the matching (Figure \ref{fig_13.01.21}(c)).

\begin{figure}[!ht]
\centering

\begin{tikzpicture}[scale=0.75]

\def\h{0};
\def\v{0};

\begin{scope} [thick, fill=black]
 \draw [fill] ($(1,0)+(\h,\v)$) circle (0.06) node[below=1pt] {{\footnotesize 1}};

 \draw [fill] ($(1.5,0)+(\h,\v)$) circle (0.01);  
 \draw [fill] ($(2,0)+(\h,\v)$) circle (0.01);  
 \draw [fill] ($(2.5,0)+(\h,\v)$) circle (0.01);  
\end{scope}

\draw (\h+1.8,-1.2) node {(a)}; 

\def\h{4};
\def\v{0};

\begin{scope}[thin]
 \draw[dashed] ($(1,0)+(\h,\v)$) -- ($(3,0)+(\h,\v)$);
 \draw ($(1,0)+(\h,\v)$) arc (180:0:1 and 3/4); 
\end{scope}

\begin{scope} [thick, fill=black]
 \draw [fill] ($(1,0)+(\h,\v)$) circle (0.06) node[below=1pt] {{\footnotesize 1}}; 
 \draw [fill] ($(2,0)+(\h,\v)$) circle (0.06) node[below=1pt] {{\footnotesize 2}}; 
 \draw [fill] ($(3,0)+(\h,\v)$) circle (0.06) node[below=1pt] {{\footnotesize 3}}; 

 \draw [fill] ($(3.5,0)+(\h,\v)$) circle (0.01);  
 \draw [fill] ($(4,0)+(\h,\v)$) circle (0.01);  
 \draw [fill] ($(4.5,0)+(\h,\v)$) circle (0.01);  
\end{scope}

\draw (\h+2.75,-1.2) node {(b)}; 

\def\h{10};
\def\v{0};

\begin{scope}[thin]

 \draw[dashed] ($(1,0)+(\h,\v)$) -- ($(4,0)+(\h,\v)$);

 \draw ($(1,0)+(\h,\v)$) arc (180:0:1 and 3/4); 
 \draw ($(2,0)+(\h,\v)$) arc (180:0:1 and 3/4);

\end{scope}

\begin{scope} [thick, fill=black]
 \draw [fill] ($(1,0)+(\h,\v)$) circle (0.06) node[below=1pt] {{\footnotesize 1}}; 
 \draw [fill] ($(2,0)+(\h,\v)$) circle (0.06) node[below=1pt] {{\footnotesize 2}}; 
 \draw [fill] ($(3,0)+(\h,\v)$) circle (0.06) node[below=1pt] {{\footnotesize 3}}; 
 \draw [fill] ($(4,0)+(\h,\v)$) circle (0.06) node[below=1pt] {{\footnotesize 4}}; 

 \draw [fill] ($(4.5,0)+(\h,\v)$) circle (0.01);  
 \draw [fill] ($(5,0)+(\h,\v)$) circle (0.01);  
 \draw [fill] ($(5.5,0)+(\h,\v)$) circle (0.01);  

\end{scope}

\draw (\h+3.25,-1.2) node {(c)}; 

\end{tikzpicture}

\caption{Possible cases of matchings in the arc diagram $\Gamma(C_n)$}\label{fig_13.01.21}
\end{figure}
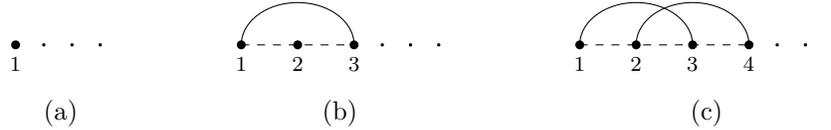

It follows from the above that $v_{n,k}$ satisfies the recurrence relation
\begin{equation}\label{22.03_1}
	v_{n+4,k+2}=v_{n+3,k+2}+v_{n+1,k+1}+v_{n,k}
\end{equation}
with the initial conditions $v_{n,k}=0$ for $k>\left\lfloor \frac{n}{2} \right\rfloor$;
$v_{n,0}=1$ for all $n$; $v_{n,1}=n-2$ for $n\geq 2$.
Consider the two-parameter generating function $v(x,t)$ for the sequence $v_{n,k}$:
$$
v(x,t)=\sum_{n=0}^{+\infty}\sum_{k=0}^{\left\lfloor \frac{n}{2}\right\rfloor}v_{n,k}x^kt^n.
$$
By multiplying (\ref{22.03_1}) by $x^{k+3}t^{n+3}$ and summing over all possible $k$ and $n$, we obtain the following equation:
\begin{equation}\label{29.03_5}
	\begin{split}
		\sum_{n=0}^{+\infty}\sum_{k=0}^{\left\lfloor\frac{n}{2}\right\rfloor}v_{n+4,k+2}x^{k+3}t^{n+3}
		=\sum_{n=0}^{+\infty}\sum_{k=0}^{\left\lfloor \frac{n}{2}\right\rfloor}v_{n+3,k+2}x^{k+3}t^{n+3}+\\
		+\sum_{n=0}^{+\infty}\sum_{k=0}^{\left\lfloor\frac{n}{2}\right\rfloor}v_{n+1,k+1}x^{k+3}t^{n+3}
		+\sum_{n=0}^{+\infty}\sum_{k=0}^{\left\lfloor \frac{n}{2}\right\rfloor}v_{n,k}x^{k+3}t^{n+3}.
	\end{split}
\end{equation}
Consider separately the formal power series on the left side of this equation.
\begin{equation*}\label{29.03_1}
	\begin{split}
	\sum_{n=0}^{+\infty}\sum_{k=0}^{\left\lfloor \frac{n}{2}\right\rfloor}v_{n+4,k+2}x^{k+3}t^{n+3}=\frac{x}{t}
	\left(\sum_{n=0}^{+\infty}\sum_{k=0}^{\left\lfloor \frac{n}{2}\right\rfloor}v_{n+4,k+2}x^{k+2}t^{n+4}\right)=\\
	=\frac{x}{t}\left(v(x,t)-\sum_{n=0}^{+\infty}v_{n,0}t^n-\sum_{n=1}^{+\infty}v_{n,1}xt^n\right)=\\
	=\frac{x}{t}\left(v(x,t)-\sum_{n=0}^{+\infty}t^n-x\sum_{n=3}^{+\infty}(n-2)t^n\right)=\\
	=\frac{x}{t}\left(v(x,t)-\frac{1}{1-t}-xt^3\left(\sum_{n=1}^{+\infty}t^n\right)'\right)=\\
	=\frac{x}{t}\left(v(x,t)-\frac{1}{1-t}-\frac{xt^3}{(1-t)^2}\right).
	\end{split}
\end{equation*}
In the same way, we get 
\begin{equation*}\label{29.03_2}
	\sum_{n=0}^{+\infty}\sum_{k=0}^{\left\lfloor \frac{n}{2}\right\rfloor}v_{n+3,k+2}x^{k+3}t^{n+3}=
	x\left(v(x,t)-\frac{1}{1-t}-\frac{xt^3}{(1-t)^2}\right).
\end{equation*}
In addition,
\begin{equation*}\label{29.03_3}
	\begin{split}
		\sum_{n=0}^{+\infty}\sum_{k=0}^{\left\lfloor \frac{n}{2}\right\rfloor}v_{n+1,k+1}x^{k+3}t^{n+3}=
		x^2t^2\sum_{n=0}^{+\infty}\sum_{k=0}^{\left\lfloor \frac{n}{2}\right\rfloor}v_{n+1,k+1}x^{k+1}t^{n+1}=\\
		x^2t^2\left(v(x,t)-\sum_{n=0}^{+\infty}v_{n,0}t^n\right)=x^2t^2\left(v(x,t)-\frac{1}{1-t}\right).
	\end{split}
\end{equation*}
Finally,
\begin{equation*}\label{29.03_4}
	\sum_{n=0}^{+\infty}\sum_{k=0}^{\left\lfloor \frac{n}{2}\right\rfloor}v_{n,k}x^{k+3}t^{n+3}=x^3t^3v(x,t).
\end{equation*}
On substituting all of the above into (\ref{29.03_5}), we get
\begin{eqnarray*}
	\frac{x}{t}\left(v(x,t)-\frac{1}{1-t}-\frac{xt^3}{(1-t)^2}\right)=
	x\left(v(x,t)-\frac{1}{1-t}-\frac{xt^3}{(1-t)^2}\right)+\\
	+x^2t^2\left(v(x,t)-\frac{1}{1-t}\right)+x^3t^3v(x,t).
\end{eqnarray*}
On solving this equation, we obtain:
\begin{equation}\label{31.03_1}
	\begin{split}
		v(x,t)=\frac{1}{1-t(1+xt^2+x^2t^3)}=\sum_{m=0}^{+\infty}t^m(1+xt^2+x^2t^3)^m=\\
		\sum_{m=0}^{+\infty}t^m\sum_{j=0}^m{m\choose j}(xt^2+x^2t^3)^j=\sum_{m=0}^{+\infty}\sum_{j=0}^m{m\choose j}x^jt^{m+2j}(1+xt)^j=\\
	  \sum_{m=0}^{+\infty}\sum_{j=0}^m{m\choose j}x^jt^{m+2j}\sum_{i=0}^j{j\choose i}(xt)^i=
		\sum_{m=0}^{+\infty}\sum_{j=0}^m\sum_{i=0}^j{m\choose j}{j\choose i}x^{j+i}t^{m+2j+i}.
	\end{split}
\end{equation}
Fix nonnegative integers $k,n$, $k\leq \left\lfloor \frac{n}{2}\right\rfloor$.
From (\ref{31.03_1}), we see that the coefficient at $x^kt^n$ is equal to $\sum\limits_i C_{n-2k+i}^{k-i}C_{k-i}^i$
over all $i$, for which the inequalities $i\geq 0$, $k-i\geq i$, $n-2k+i\geq k-i$ hold.
The last two inequalities can be rewritten as $i\leq \left\lfloor \frac{k}{2}\right\rfloor$,
$i\geq \left\lceil \frac{3k-n}{2}\right\rceil$.
Thus, for the set of acceptable values of $i$ to be nonempty,
it is necessary that $\left\lceil \frac{3k-n}{2}\right\rceil\leq \left\lfloor \frac{k}{2}\right\rfloor$.
However, according to Proposition \ref{10.06.19_1}, this condition 
is equivalent to $k\leq \left\lfloor \frac{n}{2}\right\rfloor$, except for the case when $k$ is odd and $n=2k$.
In the last case,  it is clear that the inequality $\left\lceil \frac{3k-n}{2}\right\rceil\leq \left\lfloor \frac{k}{2}\right\rfloor$ 
does not hold, and therefore, the coefficient at $x^kt^n$ is equal to zero.
This completes the proof.
\end{proof}

\begin{rmk}
Several initial values $\mu_k(\Gamma(C_n))$ are presented in Table \ref{t_14.01.21}.
The empty cells correspond to zero.
Table \ref{t_14.01.21} coincides with Table $A220074$ in \cite{OEIS} up to the sign.

\begin{table}[!ht]
\centering
{\renewcommand{\arraystretch}{0}%
\begin{tabular}{|c||c|c|c|c|c|c|c|c|c|c|c|c|c|}
 \hline\strut
 \backslashbox{$k$}{$n$}&\footnotesize $0$&\footnotesize $1$&\footnotesize $2$&\footnotesize $3$&\footnotesize $4$&\footnotesize $5$&\footnotesize $6$&
 \footnotesize $7$&\footnotesize $8$&\footnotesize $9$&\footnotesize $10$&\footnotesize $11$&\footnotesize $12$\\
 \hline
 \rule{0pt}{1.6pt}&&&&&&&&&&&\\
 \hline\strut
  \footnotesize $0$&\footnotesize $1$&\footnotesize $1$&\footnotesize $1$&\footnotesize $1$&\footnotesize $1$&\footnotesize $1$&\footnotesize $1$&\footnotesize $1$&\footnotesize $1$&\footnotesize $1$&\footnotesize $1$&\footnotesize $1$&\footnotesize $1$\\
	
 \hline\strut
  \footnotesize $1$&&&&\footnotesize $1$&\footnotesize $2$&\footnotesize $3$&\footnotesize $4$&\footnotesize $5$&\footnotesize $6$&\footnotesize $7$&\footnotesize $8$&\footnotesize $9$&\footnotesize $10$\\
 
 \hline\strut
  \footnotesize $2$&&&&&\footnotesize $1$&\footnotesize $2$&\footnotesize $4$&\footnotesize $7$&\footnotesize $11$&\footnotesize $16$&\footnotesize $22$&\footnotesize $29$&\footnotesize $37$\\
	
 \hline\strut
  \footnotesize $3$&&&&&&&&\footnotesize $2$&\footnotesize $6$&\footnotesize $13$&\footnotesize $24$&\footnotesize $40$&\footnotesize $62$\\
	
 \hline\strut
  \footnotesize $4$&&&&&&&&&\footnotesize $1$&\footnotesize $3$&\footnotesize $9$&\footnotesize $22$&\footnotesize $46$\\
	
 \hline\strut
  \footnotesize $5$&&&&&&&&&&&&\footnotesize $3$&\footnotesize $12$\\
	
 \hline\strut
  \footnotesize $6$&&&&&&&&&&&&&\footnotesize $1$\\

\hline
\end{tabular}
}
\caption{Number of $k$-edge matchings in the graph $\Gamma(C_n)$}\label{t_14.01.21}
\end{table}
\end{rmk}

Let $R$ be a commutative associative ring with $1$ and $a,b\in R$.
Consider a symmetric two-parameter Toeplitz matrix $C_{2m}(a, b)$
having the first row of the form
$$
\left(
\begin{array}{cccccc}
 0&b&a&b&\dots&b
\end{array}
\right).
$$
On substituting the value $\mu_{m-k}(\Gamma(C_{2m}))$ in (\ref{11.04.19_1}), we obtain the following theorem:
\begin{thm}
If we assume that $0^0=1$, then the hafnian of the matrix $C_{2m}(a,b)$ 
can be calculated using the following formula:
\begin{equation}\label{12.04.19_2}
\begin{split}
  &\mathrm{Hf}(C_{2m}(a,b))=\\
	&\sum_{k=p}^m (a-b)^{m-k}b^k\frac{(2k)!}{k!2^k}
	\sum\limits_{i=\max{\left(0,\left\lceil \frac{m-3k}{2}\right\rceil\right)}}^{\left\lfloor \frac{m-k}{2}\right\rfloor}
{2k+i\choose m-k-i}{m-k-i\choose i}\ ,
\end{split}
\end{equation}
where $p=0$ when $m$ is even and $p=1$ when $m$ is odd.
\end{thm}

\begin{rmk}
Equality (\ref{12.04.19_2}) allows us to calculate $\mathrm{Hf}(C_{2m}(a,b))$ in time $O(m^3)$.
\end{rmk}

\begin{ex}
Consider the matrix $C_{2m}(0,1)$. 
By calculating its hafnian using formula (\ref{12.04.19_2}) for consecutive $m$'s, we obtain the sequence:
$$
\begin{array}{c|c|c|c|c|c|c|c|c|c|c|c}
 m&1&2&3&4&5&6&7&8&9&10&\dots\\
 \hline
 \mathrm{Hf}&1&2&7&43&372&4027&51871&773186&13083385&247698481&\dots
\end{array}
$$
In terms of graph theory, its $m$-th member equals the number of perfect matchings in the arc diagram $\Gamma(C_{2m}(0,1))$ (Figure \ref{arc8}). 
In other words, this is the number of linear chord diagrams with $m$ chords such that the length of each chord does not equal $2$ (see \cite{Sullivan}). 
This sequence has the number $A265229$ in \cite{OEIS}, but  its description does not contain the  interpretation given here.

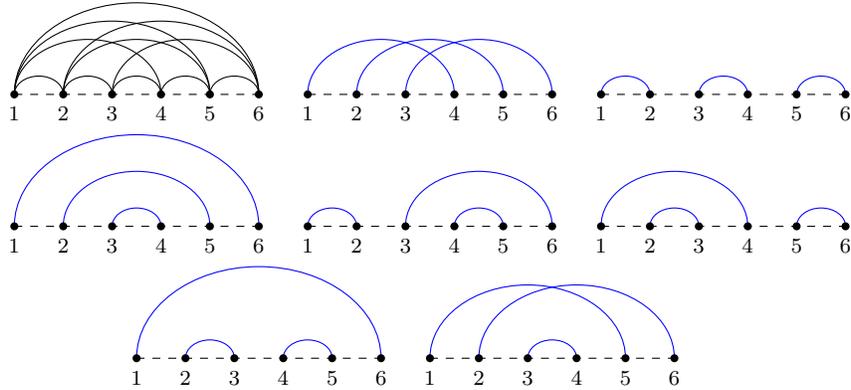
\begin{figure}[!ht]
\centering
\begin{tikzpicture}[scale=0.65]


\draw[dashed, thin] (1,0) -- (6,0);

\begin{scope}[thin]
 \draw (1,0) arc (180:0:1.5 and 9/8);
 \draw (2,0) arc (180:0:1.5 and 9/8);
 \draw (3,0) arc (180:0:1.5 and 9/8);
 \draw (1,0) arc (180:0:0.5 and 3/8);
 \draw (2,0) arc (180:0:0.5 and 3/8);
 \draw (3,0) arc (180:0:0.5 and 3/8);
 \draw (4,0) arc (180:0:0.5 and 3/8);
 \draw (5,0) arc (180:0:0.5 and 3/8);
 \draw (1,0) arc (180:0:2 and 3/2);
 \draw (2,0) arc (180:0:2 and 3/2);
 \draw (1,0) arc (180:0:2.5 and 15/8);
\end{scope}

\begin{scope} [thick, fill=black]
	\draw [fill] ($(1,0)$) circle (0.06) node[below=1pt] {{\footnotesize 1}};
	\draw [fill] ($(2,0)$) circle (0.06) node[below=1pt] {{\footnotesize 2}};
	\draw [fill] ($(3,0)$) circle (0.06) node[below=1pt] {{\footnotesize 3}};
	\draw [fill] ($(4,0)$) circle (0.06) node[below=1pt] {{\footnotesize 4}};
	\draw [fill] ($(5,0)$) circle (0.06) node[below=1pt] {{\footnotesize 5}};
	\draw [fill] ($(6,0)$) circle (0.06) node[below=1pt] {{\footnotesize 6}};
\end{scope}


\def\h{6};
\def\v{0};

\draw[dashed, thin] ($(1,0)+(\h,\v)$) -- ($(6,0)+(\h,\v)$);

\begin{scope}[thin, blue]
 \draw ($(1,0)+(\h,\v)$) arc (180:0:1.5 and 9/8);
 \draw ($(2,0)+(\h,\v)$) arc (180:0:1.5 and 9/8);
 \draw ($(3,0)+(\h,\v)$) arc (180:0:1.5 and 9/8);
\end{scope}

\begin{scope} [thick, fill=black]
	\draw [fill] ($(1,0)+(\h,\v)$) circle (0.06) node[below=1pt] {{\footnotesize 1}};
	\draw [fill] ($(2,0)+(\h,\v)$) circle (0.06) node[below=1pt] {{\footnotesize 2}};
	\draw [fill] ($(3,0)+(\h,\v)$) circle (0.06) node[below=1pt] {{\footnotesize 3}};
	\draw [fill] ($(4,0)+(\h,\v)$) circle (0.06) node[below=1pt] {{\footnotesize 4}};
	\draw [fill] ($(5,0)+(\h,\v)$) circle (0.06) node[below=1pt] {{\footnotesize 5}};
	\draw [fill] ($(6,0)+(\h,\v)$) circle (0.06) node[below=1pt] {{\footnotesize 6}};
\end{scope}


\def\h{12};
\def\v{0};

\draw[dashed, thin] ($(1,0)+(\h,\v)$) -- ($(6,0)+(\h,\v)$);

\begin{scope}[thin, blue]
 \draw ($(1,0)+(\h,\v)$) arc (180:0:0.5 and 3/8);
 \draw ($(3,0)+(\h,\v)$) arc (180:0:0.5 and 3/8);
 \draw ($(5,0)+(\h,\v)$) arc (180:0:0.5 and 3/8);
\end{scope}

\begin{scope} [thick, fill=black]
	\draw [fill] ($(1,0)+(\h,\v)$) circle (0.06) node[below=1pt] {{\footnotesize 1}};
	\draw [fill] ($(2,0)+(\h,\v)$) circle (0.06) node[below=1pt] {{\footnotesize 2}};
	\draw [fill] ($(3,0)+(\h,\v)$) circle (0.06) node[below=1pt] {{\footnotesize 3}};
	\draw [fill] ($(4,0)+(\h,\v)$) circle (0.06) node[below=1pt] {{\footnotesize 4}};
	\draw [fill] ($(5,0)+(\h,\v)$) circle (0.06) node[below=1pt] {{\footnotesize 5}};
	\draw [fill] ($(6,0)+(\h,\v)$) circle (0.06) node[below=1pt] {{\footnotesize 6}};
\end{scope}


\def\h{0};
\def\v{-2.7};

\draw[dashed, thin] ($(1,0)+(\h,\v)$) -- ($(6,0)+(\h,\v)$);

\begin{scope}[thin, blue]
 \draw ($(1,0)+(\h,\v)$) arc (180:0:2.5 and 15/8);
 \draw ($(2,0)+(\h,\v)$) arc (180:0:1.5 and 9/8);
 \draw ($(3,0)+(\h,\v)$) arc (180:0:0.5 and 3/8);
\end{scope}

\begin{scope} [thick, fill=black]
	\draw [fill] ($(1,0)+(\h,\v)$) circle (0.06) node[below=1pt] {{\footnotesize 1}};
	\draw [fill] ($(2,0)+(\h,\v)$) circle (0.06) node[below=1pt] {{\footnotesize 2}};
	\draw [fill] ($(3,0)+(\h,\v)$) circle (0.06) node[below=1pt] {{\footnotesize 3}};
	\draw [fill] ($(4,0)+(\h,\v)$) circle (0.06) node[below=1pt] {{\footnotesize 4}};
	\draw [fill] ($(5,0)+(\h,\v)$) circle (0.06) node[below=1pt] {{\footnotesize 5}};
	\draw [fill] ($(6,0)+(\h,\v)$) circle (0.06) node[below=1pt] {{\footnotesize 6}};
\end{scope}


\def\h{6};
\def\v{-2.7};

\draw[dashed, thin] ($(1,0)+(\h,\v)$) -- ($(6,0)+(\h,\v)$);

\begin{scope}[thin, blue]
 \draw ($(1,0)+(\h,\v)$) arc (180:0:0.5 and 3/8);
 \draw ($(3,0)+(\h,\v)$) arc (180:0:1.5 and 9/8);
 \draw ($(4,0)+(\h,\v)$) arc (180:0:0.5 and 3/8);
\end{scope}

\begin{scope} [thick, fill=black]
	\draw [fill] ($(1,0)+(\h,\v)$) circle (0.06) node[below=1pt] {{\footnotesize 1}};
	\draw [fill] ($(2,0)+(\h,\v)$) circle (0.06) node[below=1pt] {{\footnotesize 2}};
	\draw [fill] ($(3,0)+(\h,\v)$) circle (0.06) node[below=1pt] {{\footnotesize 3}};
	\draw [fill] ($(4,0)+(\h,\v)$) circle (0.06) node[below=1pt] {{\footnotesize 4}};
	\draw [fill] ($(5,0)+(\h,\v)$) circle (0.06) node[below=1pt] {{\footnotesize 5}};
	\draw [fill] ($(6,0)+(\h,\v)$) circle (0.06) node[below=1pt] {{\footnotesize 6}};
\end{scope}


\def\h{12};
\def\v{-2.7};

\draw[dashed, thin] ($(1,0)+(\h,\v)$) -- ($(6,0)+(\h,\v)$);

\begin{scope}[thin, blue]
 \draw ($(1,0)+(\h,\v)$) arc (180:0:1.5 and 9/8);
 \draw ($(2,0)+(\h,\v)$) arc (180:0:0.5 and 3/8);
 \draw ($(5,0)+(\h,\v)$) arc (180:0:0.5 and 3/8);
\end{scope}

\begin{scope} [thick, fill=black]
	\draw [fill] ($(1,0)+(\h,\v)$) circle (0.06) node[below=1pt] {{\footnotesize 1}};
	\draw [fill] ($(2,0)+(\h,\v)$) circle (0.06) node[below=1pt] {{\footnotesize 2}};
	\draw [fill] ($(3,0)+(\h,\v)$) circle (0.06) node[below=1pt] {{\footnotesize 3}};
	\draw [fill] ($(4,0)+(\h,\v)$) circle (0.06) node[below=1pt] {{\footnotesize 4}};
	\draw [fill] ($(5,0)+(\h,\v)$) circle (0.06) node[below=1pt] {{\footnotesize 5}};
	\draw [fill] ($(6,0)+(\h,\v)$) circle (0.06) node[below=1pt] {{\footnotesize 6}};
\end{scope}

\def\h{2.5};
\def\v{-5.4};

\draw[dashed, thin] ($(1,0)+(\h,\v)$) -- ($(6,0)+(\h,\v)$);

\begin{scope}[thin, blue]
 \draw ($(1,0)+(\h,\v)$) arc (180:0:2.5 and 15/8);
 \draw ($(2,0)+(\h,\v)$) arc (180:0:0.5 and 3/8);
 \draw ($(4,0)+(\h,\v)$) arc (180:0:0.5 and 3/8);
\end{scope}

\begin{scope} [thick, fill=black]
	\draw [fill] ($(1,0)+(\h,\v)$) circle (0.06) node[below=1pt] {{\footnotesize 1}};
	\draw [fill] ($(2,0)+(\h,\v)$) circle (0.06) node[below=1pt] {{\footnotesize 2}};
	\draw [fill] ($(3,0)+(\h,\v)$) circle (0.06) node[below=1pt] {{\footnotesize 3}};
	\draw [fill] ($(4,0)+(\h,\v)$) circle (0.06) node[below=1pt] {{\footnotesize 4}};
	\draw [fill] ($(5,0)+(\h,\v)$) circle (0.06) node[below=1pt] {{\footnotesize 5}};
	\draw [fill] ($(6,0)+(\h,\v)$) circle (0.06) node[below=1pt] {{\footnotesize 6}};
\end{scope}

\def\h{8.5};
\def\v{-5.4};

\draw[dashed, thin] ($(1,0)+(\h,\v)$) -- ($(6,0)+(\h,\v)$);

\begin{scope}[thin, blue]
 \draw ($(1,0)+(\h,\v)$) arc (180:0:2 and 3/2);
 \draw ($(2,0)+(\h,\v)$) arc (180:0:2 and 3/2);
 \draw ($(3,0)+(\h,\v)$) arc (180:0:0.5 and 3/8);
\end{scope}

\begin{scope} [thick, fill=black]
	\draw [fill] ($(1,0)+(\h,\v)$) circle (0.06) node[below=1pt] {{\footnotesize 1}};
	\draw [fill] ($(2,0)+(\h,\v)$) circle (0.06) node[below=1pt] {{\footnotesize 2}};
	\draw [fill] ($(3,0)+(\h,\v)$) circle (0.06) node[below=1pt] {{\footnotesize 3}};
	\draw [fill] ($(4,0)+(\h,\v)$) circle (0.06) node[below=1pt] {{\footnotesize 4}};
	\draw [fill] ($(5,0)+(\h,\v)$) circle (0.06) node[below=1pt] {{\footnotesize 5}};
	\draw [fill] ($(6,0)+(\h,\v)$) circle (0.06) node[below=1pt] {{\footnotesize 6}};
\end{scope}

\end{tikzpicture}

\caption{The arc diagram $\Gamma(C_6(0,1))$ and all its perfect matchings}\label{arc8}
\end{figure}
\end{ex}

\section{Hafnian of Toeplitz matrices of the second type}
As the template matrix $T_n$, now consider a symmetric Toeplitz matrix of order $n$ with the first row
$$
\left(
\begin{array}{cccccc}
 0&1&1&0&\dots&0
\end{array}
\right).
$$
We denote it by $D_n$.
This matrix is the adjacency matrix of the arc diagram shown in Figure \ref{pic_16.10_1}.

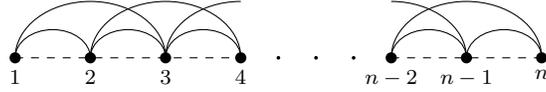
\begin{figure}[!ht]
\centering

\begin{tikzpicture}[scale=1]

\def\h{0};
\def\v{0};

\begin{scope}[thin]

 \draw[dashed] ($(1,0)+(\h,\v)$) -- ($(4,0)+(\h,\v)$);
 \draw[dashed] ($(6,0)+(\h,\v)$) -- ($(8,0)+(\h,\v)$);

 \draw ($(1,0)+(\h,\v)$) arc (180:0:1 and 3/4); 
 \draw ($(1,0)+(\h,\v)$) arc (180:0:0.5 and 3/8); 
 \draw ($(2,0)+(\h,\v)$) arc (180:0:1 and 3/4);
 \draw ($(2,0)+(\h,\v)$) arc (180:0:0.5 and 3/8); 
 \draw ($(3,0)+(\h,\v)$) arc (180:90:1 and 3/4);
 \draw ($(3,0)+(\h,\v)$) arc (180:0:0.5 and 3/8); 
 \draw ($(6,0)+(\h,\v)$) arc (180:0:1 and 3/4);
 \draw ($(6,0)+(\h,\v)$) arc (180:0:0.5 and 3/8); 
 \draw ($(7,0)+(\h,\v)$) arc (180:0:0.5 and 3/8); 
 \draw ($(7,0)+(\h,\v)$) arc (0:90:1 and 3/4);

\end{scope}

\begin{scope} [thick, fill=black]
 \draw [fill] ($(1,0)+(\h,\v)$) circle (0.06) node[below=1pt] {{\footnotesize 1}}; 
 \draw [fill] ($(2,0)+(\h,\v)$) circle (0.06) node[below=1pt] {{\footnotesize 2}}; 
 \draw [fill] ($(3,0)+(\h,\v)$) circle (0.06) node[below=1pt] {{\footnotesize 3}}; 
 \draw [fill] ($(4,0)+(\h,\v)$) circle (0.06) node[below=1pt] {{\footnotesize 4}}; 

 \draw [fill] ($(4.5,0)+(\h,\v)$) circle (0.01);  
 \draw [fill] ($(5,0)+(\h,\v)$) circle (0.01);  
 \draw [fill] ($(5.5,0)+(\h,\v)$) circle (0.01);  

 \draw [fill] ($(6,0)+(\h,\v)$) circle (0.06) node[below=1pt] {{\footnotesize $n-2$}};
 \draw [fill] ($(7,0)+(\h,\v)$) circle (0.06) node[below=1pt] {{\footnotesize $n-1$}}; 
 \draw [fill] ($(8,0)+(\h,\v)$) circle (0.06) node[below=1pt] {{\footnotesize $n$}}; 

\end{scope}

\end{tikzpicture}

\caption{The arc diagram $\Gamma(D_n)$}\label{pic_16.10_1}
\end{figure}

\begin{thm}\label{17.10_1}
Let $k$ and $n$ be nonnegative integers such that $k\leq \left\lfloor \frac{n}{2}\right\rfloor$.
Then, the number of $k$-edge matchings in the arc diagram $\Gamma(D_n)$ is equal to the following: 
\begin{equation}\label{16.10_4}
 \mu_k(\Gamma(D_n))=
\sum_{i=0}^{\min(k,\left\lfloor \frac{n-k}{2}\right\rfloor)} 
\sum\limits_{p=\max(0,i+2k-n)}^{\min(i,k-i)}
{n-k-i \choose k-p} {k-p \choose i} {i \choose p}.
\end{equation}
\end{thm}

\begin{proof}
For convenience, we use the abbreviated notation $w_{n,k}$ for  $\mu_k(\Gamma(D_n))$.
Consider a $k$-edge matching in $\Gamma(D_n)$.
If $n\geq 4$, then the following four cases are possible:
the first vertex of the diagram is not incident to the edge of the matching (Figure \ref{fig_07.06.19_1}(a));
the first and second vertices  are connected by an edge of the matching (Figure \ref{fig_07.06.19_1}(b)); 
the first vertex is incident to an edge of the matching, but the second vertex is not (Figure \ref{fig_07.06.19_1}(c)); 
the first and second vertices are incident to different edges of the matching (Figure \ref{fig_07.06.19_1}(d)).

\begin{figure}[!ht]
\centering

\begin{tikzpicture}[scale=0.75]

\def\h{0};
\def\v{0};

\begin{scope} [thick, fill=black]
 \draw [fill] ($(1,0)+(\h,\v)$) circle (0.06) node[below=1pt] {{\footnotesize 1}};

 \draw [fill] ($(1.5,0)+(\h,\v)$) circle (0.01);  
 \draw [fill] ($(2,0)+(\h,\v)$) circle (0.01);  
 \draw [fill] ($(2.5,0)+(\h,\v)$) circle (0.01);  
\end{scope}

\draw (\h+1.8,-1.2) node {(a)};

\def\h{2.5};
\def\v{0};

\begin{scope}[thin]
 \draw[dashed] ($(1,0)+(\h,\v)$) -- ($(2,0)+(\h,\v)$);
 \draw ($(1,0)+(\h,\v)$) arc (180:0:0.5 and 3/8); 
\end{scope}

\begin{scope} [thick, fill=black]
 \draw [fill] ($(1,0)+(\h,\v)$) circle (0.06) node[below=1pt] {{\footnotesize 1}}; 
 \draw [fill] ($(2,0)+(\h,\v)$) circle (0.06) node[below=1pt] {{\footnotesize 2}}; 

 \draw [fill] ($(2.5,0)+(\h,\v)$) circle (0.01);  
 \draw [fill] ($(3,0)+(\h,\v)$) circle (0.01);  
 \draw [fill] ($(3.5,0)+(\h,\v)$) circle (0.01);  
\end{scope}

\draw (\h+2.25,-1.2) node {(b)}; 
\def\h{6};
\def\v{0};

\begin{scope}[thin]
 \draw[dashed] ($(1,0)+(\h,\v)$) -- ($(3,0)+(\h,\v)$);
 \draw ($(1,0)+(\h,\v)$) arc (180:0:1 and 3/4); 
\end{scope}

\begin{scope} [thick, fill=black]
 \draw [fill] ($(1,0)+(\h,\v)$) circle (0.06) node[below=1pt] {{\footnotesize 1}}; 
 \draw [fill] ($(2,0)+(\h,\v)$) circle (0.06) node[below=1pt] {{\footnotesize 2}}; 
 \draw [fill] ($(3,0)+(\h,\v)$) circle (0.06) node[below=1pt] {{\footnotesize 3}}; 

 \draw [fill] ($(3.5,0)+(\h,\v)$) circle (0.01);  
 \draw [fill] ($(4,0)+(\h,\v)$) circle (0.01);  
 \draw [fill] ($(4.5,0)+(\h,\v)$) circle (0.01);  
\end{scope}

\draw (\h+2.75,-1.2) node {(c)}; 

\def\h{10.5};
\def\v{0};

\begin{scope}[thin]

 \draw[dashed] ($(1,0)+(\h,\v)$) -- ($(4,0)+(\h,\v)$);

 \draw ($(1,0)+(\h,\v)$) arc (180:0:1 and 3/4); 
 \draw ($(2,0)+(\h,\v)$) arc (180:0:1 and 3/4);

\end{scope}

\begin{scope} [thick, fill=black]
 \draw [fill] ($(1,0)+(\h,\v)$) circle (0.06) node[below=1pt] {{\footnotesize 1}}; 
 \draw [fill] ($(2,0)+(\h,\v)$) circle (0.06) node[below=1pt] {{\footnotesize 2}}; 
 \draw [fill] ($(3,0)+(\h,\v)$) circle (0.06) node[below=1pt] {{\footnotesize 3}}; 
 \draw [fill] ($(4,0)+(\h,\v)$) circle (0.06) node[below=1pt] {{\footnotesize 4}}; 

 \draw [fill] ($(4.5,0)+(\h,\v)$) circle (0.01);  
 \draw [fill] ($(5,0)+(\h,\v)$) circle (0.01);  
 \draw [fill] ($(5.5,0)+(\h,\v)$) circle (0.01);  

\end{scope}

\draw (\h+3.25,-1.2) node {(d)}; 

\end{tikzpicture}

\caption{Possible cases of matchings in the arc diagram $\Gamma(D_n)$}\label{fig_07.06.19_1}
\end{figure}
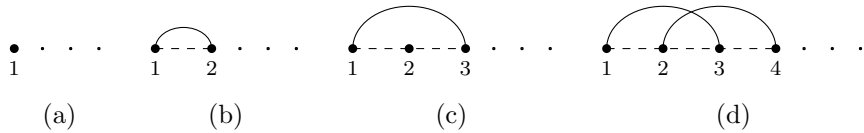

It follows from the above that $w_{n,k}$ satisfies the recurrence relation
\begin{equation}\label{16.10_1}
	w_{n+4,k+2}=w_{n+3,k+2}+w_{n+2,k+1}+w_{n+1,k+1}+w_{n,k}
\end{equation}
with the initial conditions $w_{n,k}=0$ for $k>\left\lfloor \frac{n}{2} \right\rfloor$;
$w_{n,0}=1$ for all $n$; $w_{n,1}=2n-3$ for $n\geq 2$.
Consider the two-parameter generating function for the sequence $w_{n,k}$:
$$
w(x,t)=\sum_{n=0}^{+\infty}\sum_{k=0}^{\left\lfloor \frac{n}{2}\right\rfloor}w_{n,k}x^kt^n.
$$
On multiplying  (\ref{16.10_1}) by $x^{k+3}t^{n+3}$ and summing over all possible $k$ and $n$, 
we get the following equation:
\begin{equation}\label{16.10_2}
	\begin{split}
		\sum_{n=0}^{+\infty}\sum_{k=0}^{\left\lfloor\frac{n}{2}\right\rfloor}w_{n+4,k+2}x^{k+3}t^{n+3}
		=\sum_{n=0}^{+\infty}\sum_{k=0}^{\left\lfloor \frac{n}{2}\right\rfloor}w_{n+3,k+2}x^{k+3}t^{n+3}+\\
		\sum_{n=0}^{+\infty}\sum_{k=0}^{\left\lfloor\frac{n}{2}\right\rfloor}w_{n+2,k+1}x^{k+3}t^{n+3}
		+\sum_{n=0}^{+\infty}\sum_{k=0}^{\left\lfloor\frac{n}{2}\right\rfloor}w_{n+1,k+1}x^{k+3}t^{n+3}+\\
		\sum_{n=0}^{+\infty}\sum_{k=0}^{\left\lfloor \frac{n}{2}\right\rfloor}w_{n,k}x^{k+3}t^{n+3}.
	\end{split}
\end{equation}
On solving this equation, we obtain:
\begin{equation}\label{16.10_3}
	\begin{split}
		w(x,t)=&\frac{1}{1-t(1+xt+xt^2+x^2t^3)}=\\
		&\sum_{m=0}^{+\infty}\sum_{j=0}^m\sum_{i=0}^j\sum_{p=0}^i{m \choose j}{j \choose i}{i \choose p}x^{j+p}t^{m+j+i+p}.
	\end{split}
\end{equation}
Fix nonnegative integers $k$ and $n\geq 2k$.
From (\ref{16.10_3}), we see that the coefficient at $x^kt^n$ is equal to the sum 
$\sum\limits_{i}\sum\limits_p {n-k-i \choose k-p}{k-p \choose i}{i \choose p}$, 
over all $i,p$ for which the inequalities $i\geq p\geq 0$, $k-p\geq i$, $n-k-i\geq k-p$ hold.
It can be shown that this system of inequalities is equivalent 
to the following system of inequalities:
$0\leq i\leq \min(k,\left\lfloor\frac{n-k}{2} \right\rfloor)$,
$\max(0,i+2k-n)\leq p\leq \min(i,k-i)$.
This completes the proof.
\end{proof}

\begin{rmk}
Note that the arc diagram $\Gamma(D_n)$ is isomorphic to the triangular lattice shown in Figure \ref{pic_17.10_1}.
Thus, formula (\ref{16.10_4}) also allows us to calculate the number of $k$-edge matchings in such lattices. 

\begin{figure}[!ht]
\centering

\begin{tikzpicture}[scale=0.9]

\def\h{0}; 
\def\v{0};

\begin{scope}[thin]

 \draw ($(1,0)+(\h,\v)$) -- ($(2,1)+(\h,\v)$);
 \draw ($(2,0)+(\h,\v)$) -- ($(3,1)+(\h,\v)$);
 \draw ($(1,0)+(\h,\v)$) -- ($(3.5,0)+(\h,\v)$);
 \draw ($(1,1)+(\h,\v)$) -- ($(3.5,1)+(\h,\v)$);
 \draw ($(1,0)+(\h,\v)$) -- ($(1,1)+(\h,\v)$);
 \draw ($(2,0)+(\h,\v)$) -- ($(2,1)+(\h,\v)$);
 \draw ($(3,0)+(\h,\v)$) -- ($(3,1)+(\h,\v)$);
 \draw ($(4.5,0)+(\h,\v)$) -- ($(6,0)+(\h,\v)$);
 \draw ($(5,0)+(\h,\v)$) -- ($(5,1)+(\h,\v)$);
 \draw ($(4.5,1)+(\h,\v)$) -- ($(6,1)+(\h,\v)$);
 \draw ($(6,0)+(\h,\v)$) -- ($(6,1)+(\h,\v)$);
 \draw ($(5,0)+(\h,\v)$) -- ($(6,1)+(\h,\v)$);

\end{scope}

\begin{scope} [thick, fill=black]
 \draw [fill] ($(1,0)+(\h,\v)$) circle (0.06) node[below=1pt] {{\footnotesize 2}}; 
 \draw [fill] ($(1,1)+(\h,\v)$) circle (0.06) node[above=1pt] {{\footnotesize 1}}; 
 \draw [fill] ($(2,0)+(\h,\v)$) circle (0.06) node[below=1pt] {{\footnotesize 4}}; 
 \draw [fill] ($(2,1)+(\h,\v)$) circle (0.06) node[above=1pt] {{\footnotesize 3}}; 
 \draw [fill] ($(3,0)+(\h,\v)$) circle (0.06) node[below=1pt] {{\footnotesize 6}}; 
 \draw [fill] ($(3,1)+(\h,\v)$) circle (0.06) node[above=1pt] {{\footnotesize 5}}; 

 \draw [fill] ($(3.5,0)+(\h,\v+0.5)$) circle (0.01);  
 \draw [fill] ($(4,0)+(\h,\v+0.5)$) circle (0.01);  
 \draw [fill] ($(4.5,0)+(\h,\v+0.5)$) circle (0.01);  

 \draw [fill] ($(5,0)+(\h,\v)$) circle (0.06) node[below=1pt] {{\footnotesize $n-2$}};
 \draw [fill] ($(5,1)+(\h,\v)$) circle (0.06) node[above=1pt] {{\footnotesize $n-3$}}; 
 \draw [fill] ($(6,0)+(\h,\v)$) circle (0.06) node[below=1pt] {{\footnotesize $n$}}; 
 \draw [fill] ($(6,1)+(\h,\v)$) circle (0.06) node[above=1pt] {{\footnotesize $n-1$}}; 

\end{scope}

\draw (\h+3.5,-1.2) node {(a)}; 

\def\h{7}; 
\def\v{0};

\begin{scope}[thin]

 \draw ($(1,0)+(\h,\v)$) -- ($(2,1)+(\h,\v)$);
 \draw ($(2,0)+(\h,\v)$) -- ($(3,1)+(\h,\v)$);
 \draw ($(1,0)+(\h,\v)$) -- ($(3.5,0)+(\h,\v)$);
 \draw ($(1,1)+(\h,\v)$) -- ($(3.5,1)+(\h,\v)$);
 \draw ($(1,0)+(\h,\v)$) -- ($(1,1)+(\h,\v)$);
 \draw ($(2,0)+(\h,\v)$) -- ($(2,1)+(\h,\v)$);
 \draw ($(3,0)+(\h,\v)$) -- ($(3,1)+(\h,\v)$);
 \draw ($(5,0)+(\h,\v)$) -- ($(5,1)+(\h,\v)$);
 \draw ($(4.5,1)+(\h,\v)$) -- ($(6,1)+(\h,\v)$);
 \draw ($(5,0)+(\h,\v)$) -- ($(6,1)+(\h,\v)$);
 \draw ($(4.5,0)+(\h,\v)$) -- ($(5,0)+(\h,\v)$);

\end{scope}

\begin{scope} [thick, fill=black]
 \draw [fill] ($(1,0)+(\h,\v)$) circle (0.06) node[below=1pt] {{\footnotesize 2}}; 
 \draw [fill] ($(1,1)+(\h,\v)$) circle (0.06) node[above=1pt] {{\footnotesize 1}}; 
 \draw [fill] ($(2,0)+(\h,\v)$) circle (0.06) node[below=1pt] {{\footnotesize 4}}; 
 \draw [fill] ($(2,1)+(\h,\v)$) circle (0.06) node[above=1pt] {{\footnotesize 3}}; 
 \draw [fill] ($(3,0)+(\h,\v)$) circle (0.06) node[below=1pt] {{\footnotesize 6}}; 
 \draw [fill] ($(3,1)+(\h,\v)$) circle (0.06) node[above=1pt] {{\footnotesize 5}}; 

 \draw [fill] ($(3.5,0)+(\h,\v+0.5)$) circle (0.01);  
 \draw [fill] ($(4,0)+(\h,\v+0.5)$) circle (0.01);  
 \draw [fill] ($(4.5,0)+(\h,\v+0.5)$) circle (0.01);  

 \draw [fill] ($(5,0)+(\h,\v)$) circle (0.06) node[below=1pt] {{\footnotesize $n-1$}};
 \draw [fill] ($(5,1)+(\h,\v)$) circle (0.06) node[above=1pt] {{\footnotesize $n-2$}}; 
 \draw [fill] ($(6,1)+(\h,\v)$) circle (0.06) node[above=1pt] {{\footnotesize $n$}}; 

\end{scope}

\draw (\h+3.5,-1.2) node {(b)}; 

\end{tikzpicture}

\caption{The triangular lattice $\Gamma(D_n)$: (a) $n$ is even; (b) $n$ is odd}\label{pic_17.10_1}
\end{figure}
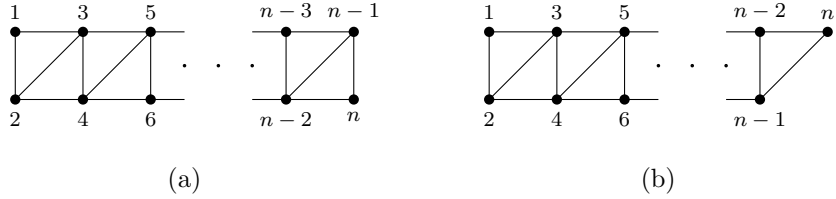
\end{rmk}

\begin{rmk}
Several initial values $\mu_k(\Gamma(D_n))$ are presented in Table \ref{t_10.01.21}.
The empty cells correspond to zero.
Note that the sequence of the first nonzero elements in the rows is the Fibonacci sequence,
the sequence of the second nonzero elements in the rows has the number $A023610$ in \cite{OEIS},
and nonzero elements of the third row coincide with the sequence $A130883$, excluding the starting element.

\begin{table}[!ht]
\centering
{\renewcommand{\arraystretch}{0}%
\begin{tabular}{|c||c|c|c|c|c|c|c|c|c|c|c|c|c|}
 \hline\strut
 \backslashbox{$k$}{$n$}&\footnotesize $0$&\footnotesize $1$&\footnotesize $2$&\footnotesize $3$&\footnotesize $4$&\footnotesize $5$&\footnotesize $6$&
 \footnotesize $7$&\footnotesize $8$&\footnotesize $9$&\footnotesize $10$&\footnotesize $11$&\footnotesize $12$\\
 \hline
 \rule{0pt}{1.6pt}&&&&&&&&&&&\\
 \hline\strut
  \footnotesize $0$&\footnotesize $1$&\footnotesize $1$&\footnotesize $1$&\footnotesize $1$&\footnotesize $1$&\footnotesize $1$&\footnotesize $1$&\footnotesize $1$&\footnotesize $1$&\footnotesize $1$&\footnotesize $1$&\footnotesize $1$&\footnotesize $1$\\
	
 \hline\strut
  \footnotesize $1$&&&\footnotesize $1$&\footnotesize $3$&\footnotesize $5$&\footnotesize $7$&\footnotesize $9$&\footnotesize $11$&\footnotesize $13$&\footnotesize $15$&\footnotesize $17$&\footnotesize $19$&\footnotesize $21$\\
 
 \hline\strut
  \footnotesize $2$&&&&&\footnotesize $2$&\footnotesize $7$&\footnotesize $16$&\footnotesize $29$&\footnotesize $46$&\footnotesize $67$&\footnotesize $92$&\footnotesize $121$&\footnotesize $154$\\
	
 \hline\strut
  \footnotesize $3$&&&&&&&\footnotesize $3$&\footnotesize $15$&\footnotesize $43$&\footnotesize $95$&\footnotesize $179$&\footnotesize $303$&\footnotesize $475$\\
	
 \hline\strut
  \footnotesize $4$&&&&&&&&&\footnotesize $5$&\footnotesize $30$&\footnotesize $104$&\footnotesize $271$&\footnotesize $591$\\
	
 \hline\strut
  \footnotesize $5$&&&&&&&&&&&\footnotesize $8$&\footnotesize $58$&\footnotesize $235$\\
	
 \hline\strut
  \footnotesize $6$&&&&&&&&&&&&&\footnotesize $13$\\

\hline
\end{tabular}
}
\caption{Number of $k$-edge matchings in the graph $\Gamma(D_n)$}\label{t_10.01.21}
\end{table}
\end{rmk}

Let $R$ be a commutative associative ring with $1$ and $a,b\in R$.
Consider a symmetric two-parameter Toeplitz matrix $D_{2m}(a, b)$
having the first row in the form
$$
\left(
\begin{array}{cccccc}
 0&a&a&b&\dots&b
\end{array}
\right).
$$
On substituting the value $\mu_{m-k}(\Gamma(D_{2m}))$ in (\ref{11.04.19_1}), we obtain the following theorem:
\begin{thm}
If we assume that $0^0=1$, then the hafnian of the matrix $D_{2m}(a, b)$ is expressed using the following formula: 
\begin{equation}\label{17.10_3}
  \mathrm{Hf}(D_{2m}(a,b))=\sum_{k=0}^m (a-b)^{m-k}b^k\frac{(2k)!}{k!2^k}\mu_{m-k}(\Gamma(D_{2m}))\ ,
\end{equation}
where
\begin{equation*}
\begin{split}
&\mu_{m-k}(\Gamma(D_{2m}))=\\
&\sum_{i=0}^{\min(m-k,\left\lfloor \frac{m+k}{2}\right\rfloor)} 
\sum\limits_{p=\max(0,i-2k)}^{\min(i,m-k-i)}
{m+k-i \choose m-k-p} {m-k-p \choose i} {i \choose p}.
\end{split}
\end{equation*}
\end{thm}

\begin{rmk}
Equality (\ref{17.10_3}) allows us to calculate $\mathrm{Hf}(D_{2m}(a,b))$ in time $O(m^4)$.
\end{rmk}

\begin{ex}
Consider the matrix $D_{2m}(0,1)$. 
By calculating its hafnian using formula (\ref{17.10_3}) for consecutive $m$s, we obtain the sequence:
$$
\begin{array}{c|c|c|c|c|c|c|c|c|c|c|c|}
 m&1&2&3&4&5&6&7&8&9&10&\dots\\
 \hline
 \mathrm{Hf}&0&0&1&10&99&1146&15422&237135&4106680&79154927&\dots
\end{array}
$$
In terms of graphs, the $m$-th member is equal to the number of perfect matchings in the arc diagram $\Gamma(D_{2m}(0,1))$ (Figure \ref{arc9}). 
In other words, this is the number of linear chord diagrams with $m$ chords such that the length of every chord is at least $3$ (see also \cite{Sullivan}, \cite{Cameron}). 
This sequence has the number $A190823$ in \cite{OEIS}.

\begin{figure}[!ht]
\centering
\begin{tikzpicture}[scale=0.65]


\draw[dashed, thin] (1,0) -- (6,0);

\begin{scope}[thin]
 \draw (1,0) arc (180:0:1.5 and 9/8);
 \draw (2,0) arc (180:0:1.5 and 9/8);
 \draw (3,0) arc (180:0:1.5 and 9/8);
 \draw (1,0) arc (180:0:2 and 3/2);
 \draw (2,0) arc (180:0:2 and 3/2);
 \draw (1,0) arc (180:0:2.5 and 15/8);
\end{scope}

\begin{scope} [thick, fill=black]
	\draw [fill] ($(1,0)$) circle (0.06) node[below=1pt] {{\footnotesize 1}};
	\draw [fill] ($(2,0)$) circle (0.06) node[below=1pt] {{\footnotesize 2}};
	\draw [fill] ($(3,0)$) circle (0.06) node[below=1pt] {{\footnotesize 3}};
	\draw [fill] ($(4,0)$) circle (0.06) node[below=1pt] {{\footnotesize 4}};
	\draw [fill] ($(5,0)$) circle (0.06) node[below=1pt] {{\footnotesize 5}};
	\draw [fill] ($(6,0)$) circle (0.06) node[below=1pt] {{\footnotesize 6}};
\end{scope}


\def\h{7};
\def\v{0};

\draw[dashed, thin] ($(1,0)+(\h,\v)$) -- ($(6,0)+(\h,\v)$);

\begin{scope}[thin, blue]
 \draw ($(1,0)+(\h,\v)$) arc (180:0:1.5 and 9/8);
 \draw ($(2,0)+(\h,\v)$) arc (180:0:1.5 and 9/8);
 \draw ($(3,0)+(\h,\v)$) arc (180:0:1.5 and 9/8);
\end{scope}

\begin{scope} [thick, fill=black]
	\draw [fill] ($(1,0)+(\h,\v)$) circle (0.06) node[below=1pt] {{\footnotesize 1}};
	\draw [fill] ($(2,0)+(\h,\v)$) circle (0.06) node[below=1pt] {{\footnotesize 2}};
	\draw [fill] ($(3,0)+(\h,\v)$) circle (0.06) node[below=1pt] {{\footnotesize 3}};
	\draw [fill] ($(4,0)+(\h,\v)$) circle (0.06) node[below=1pt] {{\footnotesize 4}};
	\draw [fill] ($(5,0)+(\h,\v)$) circle (0.06) node[below=1pt] {{\footnotesize 5}};
	\draw [fill] ($(6,0)+(\h,\v)$) circle (0.06) node[below=1pt] {{\footnotesize 6}};
\end{scope}

\end{tikzpicture}

\caption{The arc diagram $\Gamma(D_6(0,1))$ and its only perfect matching}\label{arc9}
\end{figure}
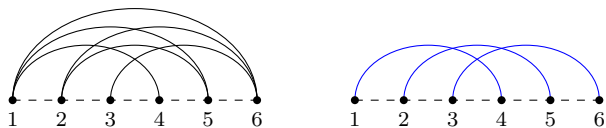

\end{ex}


\begin{thebibliography}{99}
\bibitem{Titan}
{\it  Bj\"orklund A., Gupt B., Quesada N.\/} {A faster hafnian formula for complex matrices
and its benchmarking on a supercomputer} // ACM Journal of Experimental Algorithmics. 2019. V. 24. 1.11. 

\bibitem{Grimson}
{\it Grimson R.C.\/} {Enumeration of Dimer (Domino) Configurations} //
Discrete Mathematics. 1977. V. 18. pp. 167--178. 

\bibitem{Krasko} 
{\it Krasko E., Omelchenko A.\/} Enumeration of chord diagrams without loops and parallel chords
// The Electronic Journal of Combinatorics. 2017. 24(3). \#3.43. 

\bibitem{Sullivan}
{\it Sullivan E.\/} {Linear chord diagrams with long chords} //
The Electronic Journal of Combinatorics. 2017. V. 24(4). pp. 1--8. 

\bibitem{Efimov}
{\it Efimov D.B.\/} The hafnian and a commutative analogue of the Grassmann algebra //
Electronic Journal of Linear Algebra. 2018. V. 34. pp. 54-60. 


\bibitem{Zaslavsky}
{\it Zaslavsky T.\/}  {Complementary Matching Vectors and the Uniform Matching Extension Property} //
European Journal of Combinatorics. 1981. V. 2. pp. 91--103. 
 
\bibitem{Young}
{\it Young D.\/} {The Number of Domino Matchings in the Game of Memory} //
Journal of Integer Sequences. 2018. V. 21(8). pp. 1--14.

\bibitem{Efimov2}
{\it Efimov D.B.} {The hafnian of Toeplitz matrices of special type, perfect matchings and Bessel polynomials} //
Bulletin of Syktyvkar University. Series 1: Mathematics. Mechanics. Informatics. 2018. V. 3(28). pp. 56--64. (in Russian).

\bibitem{OEIS}
{\it  N.J.A. Sloane, editor\/} The On-Line Encyclopedia of Integer Sequences, published electronically at \url{https://oeis.org}.

\bibitem{Cameron}
{\it Cameron N.T., Killpatrick K.\/} {Statistics on Linear Chord Diagrams} //
Discrete Mathematics and Theoretical Computer Science. 2019. V. 21(2). pp. 1--10. 


\end{thebibliography}
\end{document}